\documentclass[a4paper,12pt]{article}
\usepackage{amssymb,amsmath,amsfonts}
\usepackage{enumerate,amsmath,amsfonts,bm,float}
\usepackage{latexsym,amssymb,amsthm, amsfonts, euscript}
\usepackage{amscd,enumerate,verbatim,calc}
\usepackage{longtable, scalefnt}
\newtheorem{Exaexa}{Example}[section]

\numberwithin{equation}{section}
\newtheorem{thm}{Theorem}[section]
\theoremstyle{definition}
\newtheorem{dfn}{Definition}[section]

\newtheorem{rem}{Remark}[section]
\theoremstyle{plain}
\newtheorem{lem}[thm]{Lemma}
\newtheorem{prop}{Proposition}[section]
\newtheorem{cor}{Corollary}[section]
\begin{document}
\begin{center}
{\large \textbf{Weight distributions of  all irreducible $\mu$-constacyclic codes of length $\ell^n$}}\\\medskip\medskip
Manjit Singh \\
E-mail: manjitsingh.math@gmail.com \\
 Department of Mathematics,\\ Deenbandhu Chhotu Ram University of Science 
and Technology, Murthal-131039, Sonepat, India.\medskip \\
\end{center}
\begin{abstract}  Let $\mathbb{F}_q$ be a finite field of order $q$ and integer $n\ge 1$. Let $\ell$ be a  prime such that $\ell^k|(q-1)$ for some integer $k\ge 1$ and  $\mu$ be an element of order $\ell^k$ in $\mathbb{F}_q$.  In this paper, we determine the weight distributions of all irreducible $\mu$-constacyclic codes of length $\ell^n$ over $\mathbb{F}_q$. Explicit expressions for the generator polynomials and  codewords of these codes are also obtained. 
\end{abstract}
\section{Introduction}

Throughout in this paper, $\mathbb{F}_q$ will denote a finite field with $q$ elements. For any integer $m\ge1$, $\mathbb{F}_q^m=\{(c_1,c_2,\cdots, c_m):c_i\in \mathbb{F}_q, 1\le i\le m\}$ is a linear space over $\mathbb{F}_q$. A linear $[m,k]_q$ code of length $m$ over $\mathbb{F}_q$ is an $k$-dimensional subspace of $\mathbb{F}_q^m$ with $q^k$  codewords.   The Hamming distance between two codewords ${{c}}=(c_1,c_2,\cdots,c_m)$ and ${{c^\prime}}=(c_1^{\prime},c_2^{\prime},\cdots,c_m^{\prime})$ in $\mathbb{F}_q^m$ is the number of components $i$ such that $c_i\neq c_i^{\prime}$. The Hamming weight of a vector ${{c}}=(c_1,c_2,\cdots,c_m)$ is the number of nonzero components of ${{c}}$. The minimum Hamming distance or simply the minimum distance  of a linear code is also equal to the minimum Hamming weight over all nonzero codewords (see \cite{huff,ling04}).

Let $A_i$  be the number of codewords  of  a linear code $C$ of length $m$ over $\mathbb{F}_q$ with Hamming weight $i$, where $0\le i\le m$. The sequence $(1, A_1, A_2, \cdots, A_m)$ is called the weight distributions of the code $C$.  Note that $A_{0}=1$ and $A_{i}=0$ for all $1\le i<d$, where $d$ is the minimum Hamming distance of the code. The weight distributions for irreducible cyclic codes has been studied by many authors (see \cite{chen15,dinh, sharma, vander, vega, zhu}).  

 Constacyclic codes play a very significant role in the theory of error-correcting codes and  are  also  known for their rich algebraic structure. Constacyclic codes can be efficiently encoded using shift registers, which explains their preferred
role in engineering. Over the last few decades, much has been written about constacyclic codes \cite{ay, chen12,  dinh13, li,raka}. In \cite{zhu}, Zhu et al. obtained the weight distribution of a class of cyclic codes of length $\ell^m$, where $\ell $ is a prime satisfying that $\ell^v||(q-1)$ and $v$ is a positive integer, and that $4|(q-1)$ if $\ell=2$. Chen et al. \cite{chen12}  introduced an equivalence relation, called isometry, to classify constacyclic codes over finite fields. They further characterized generator polynomials of constacyclic codes of length $\ell^tp^s$ over $\mathbb{F}_{p^n}$, where $\ell$, $p$ are distinct primes and $s$, $t$, $n$ are positive integers. Recently,   Li and Yue \cite{li} obtained  the weight distributions of all irreducible constacyclic codes and their duals of length $\ell^n$ over $\mathbb{F}_q$ when $\ell^v||(q+1)$, where $\ell$ is prime satisfying $\gcd(\ell, q(q-1))=1$.   Here it is important to note that the weight distributions of all irreducible constacyclic codes of length $\ell^n$ over $\mathbb{F}_q$ are not considered so far when $\ell^v||(q-1)$, where $\ell$ is prime satisfying $\gcd(\ell, q)=1$.

Let $\ell$ be a prime such that $\ell^k|(q-1)$ for some integer $k\ge 1$ and   $\mu$ be a primitive $\ell^k$th root of unity in $\mathbb{F}_q$, where $4|(q-1)$ if $\ell=2$. The object of this paper is to determine the weight distributions of  all irreducible $\mu$-constacyclic codes of length $\ell^n$  over $\mathbb{F}_q$. 

 We introduce some preliminaries and auxiliary results in Section 2 and prove our main results in Section 4.  In Section 3, we revisit the irreducible factorization of $x^{\ell^n}-\mu$ over $\mathbb{F}_q$. The weight distributions of all irreducible $\mu$-constacyclic codes of length $\ell^n$ are determined in Section 4.  The explicit expressions of generator polynomials and  codewords of these codes are also obtained. In addition, two numerical examples are given to validate the results.

\section{Preliminaries}  
 Let $\mu$ be a nonzero element in $\mathbb{F}_q$ and integer $m\ge1$. A linear  code $C$ of length $m$ over $\mathbb{F}_q$ is called a $\mu$-constacyclic code if   \begin{center}$( c_1,c_2,\cdots,
c_{m-1},0)+c_0(0,0,\cdots,0,{\mu})\in C$  $\forall$  $(c_0, c_1,c_2,\cdots,
c_{m-1})\in C$.\end{center} The  code $C$ is  called cyclic  if $\mu=1$ and  negacyclic  if $\mu=-1$. If $\gcd(m,q)=1$, a $\mu$-constacyclic code of length $m$ is called simple-root code; otherwise it is called repeated-root code.  Identify a vector $(a_0,a_1,\cdots, a_{m-1})\in\mathbb{F}_q^m$ with a polynomial $a_0x^{m-1}+a_1x^{m-2}+\cdots+a_{m-1}$ modulo $(x^m-\mu)$ in  $\mathbb{F}_q[x]$ and $\gcd(m,q)=1$.  A simple-root $\mu$-constacyclic code $C$  can be represented as an ideal of the quotient ring  
 $\mathbb{F}_q[x]/\langle x^m-\mu\rangle$. It is well-known that each ideal in $\mathbb{F}_q[x]/\langle x^m-\mu\rangle$ is of the form $\langle g(x)\rangle$, where $g(x)$ is a monic divisor of $x^m-\mu$ in $\mathbb{F}_q[x]$. The polynomial $g(x)$ is known as the generator polynomial of the code $C$. An $\mu$-constacyclic code $C$ is called irreducible or minimal over $\mathbb{F}_q$ if $h(x)=(x^m-\mu)/g(x)$ is an irreducible polynomial over $\mathbb{F}_q$, where the polynomial $h(x)$ is known as the parity check polynomial or simply check polynomial of $C$.
\begin{dfn}[see \cite{ling04}] Two $(m,M)$-codes, where $M$ is the size of the code, over $\mathbb{F}_q$ are equivalent if one can be obtained from the other by a combination of operations of the following type:\begin{enumerate}
\item[(i)] permutation of the $m$ digits of the codewords;
\item[(ii)] multiplication of the symbols appearing in a fixed position by a nonzero scalar.\end{enumerate}\end{dfn}

\begin{prop}\cite[Theorem 7.5]{bla75}\label{qchain}  Let $\eta_1$ and $\eta_2$ be any two primitive $n$th roots of unity in a suitable extension field of $\mathbb{F}_q$. Then if $\Omega$ is a union of $q$ chains mod $n$, the polynomials \begin{center}$g_{1}(x)=\prod_{i\in \Omega}(x-\eta_1^i)$ and $g_{2}(x)=\prod_{i\in \Omega}(x-\eta_2^i)$ \end{center} generate equivalent codes. \end{prop}
\begin{lem}\cite[p.72]{ling04}\label{cc} Equivalent linear codes  have the same length, dimension
and distance.  \end{lem}

In literature, the irreducibility of  binomials  over $\mathbb{F}_q$ has been completely characterized early in 1866 (Serret) in the following way. 
\begin{lem}\cite[Theorem 3.75]{lidl97}  \label{bir}  Let $t\ge 2$ be an integer and  $a\in\mathbb{F}_q^*$. Then the binomial $x^{t}-a$ is irreducible in $\mathbb{F}_q[x]$ if and only if the following two conditions are satisfied: (i) each prime factor of $t$ divides the order $e$ of $a$ in $\mathbb{F}_q^*$, but  not  $(q-1)/e$; (ii) $q\equiv1\pmod 4$ if $t\equiv0\pmod 4$.
\end{lem}
\begin{lem}\cite[Theorem 3.1]{chen15}\label{chenfact} Suppose that $q-1=l^vc$, where $v>0$, $l$ is a prime, $\gcd(l,c)=1$ and $4|(q-1)$ if $l=2$. Then there is an irreducible factorization over $\mathbb{F}_q$:
\begin{eqnarray*}x^{\ell^n}-1=\prod_{k=0}^{l^n-1}(x-\eta^k)\prod_{a=1}^{n-v}\prod_{\substack{i=0\\\gcd(l,i)=1}}^{l^v-1}(x^{l^a}-\eta^i)\end{eqnarray*} where $\eta$ is a primitive $l^v$th root of unity in $\mathbb{F}_q$ and $n>v$.\end{lem}
\section{Explicit factorization of $x^{\ell^n}-\mu\in\mathbb{F}_q[x]$}

 For each integer $n\ge1$ and $\ell$ is a prime,  $\nu_\ell(n)$  denotes the maximum power of $\ell$ that divides $n$. For any odd prime power $q$, let $s=\nu_\ell(q-1)$. Observe that if $\gcd(\ell,q-1)=1$, then $s=0$; otherwise  $\ell^k|(q-1)$ for all $1\le k\le s$. Let $\mu_k$ be a primitive $\ell^k$th root of unity in $\mathbb{F}_q^*$.   Consider  the binomial $x^{\ell^n}-\mu_k\in\mathbb{F}_q[x]$ for $1\le k\le s$ and integer $n\ge 1$.   Note that for $k=0$, the factorization of $x^{\ell^n}-\mu_k$ over $\mathbb{F}_q$ is given in Lemma \ref{chenfact}.

 Although the following theorem follows immediately from
\cite[Theorem 4.1(ii)]{chen12}, however, for the sake of completeness of this paper, we rephrase and  also provide a proof of it.
\begin{thm}\label{bifactq} Let $\gcd(q,\ell)=1$, $r=\min\{n,s-k\}$, $1\le k\le s$, and $s\ne 1$ if $\ell=2$.  Then, for every $n\ge1$, the irreducible factorization of $x^{\ell^n}-\mu_k$ is given by  \begin{eqnarray*}x^{\ell^n}-\mu_k=
\prod_{i=1}^{\ell^{r}}(x^{\ell^{n-r}}-\mu_{k+r}^{\ell^ki+1}).\end{eqnarray*}\end{thm}
\begin{proof} Since $r=\min\{n,s-k\}$ and $1\le k\le s$, so  $\ell^r|(q-1)$ and hence the element $\mu_r\in\mathbb{F}_q^*$. Thus $\mathbb{F}_q$ is  the minimal splitting field of $x^{\ell^r}-1$ over itself. It follows that
\begin{eqnarray}\label{fact2r}
x^{\ell^r}-1=\prod_{i=1}^{\ell^r}(x-\mu_r^i).
\end{eqnarray} Since $n\ge r$,  the factorization of $x^{\ell^n}-\mu_k$ can be expressed as:  \begin{eqnarray*}
x^{\ell^n}-\mu_k=(x^{\ell^{n-r}})^{\ell^r}-\mu_{k+r}^{\ell^r}=
\mu_{k+r}^{\ell^r}\bigg(\bigg(\frac{x^{\ell^{n-r}}}{\mu_{k+r}}\bigg)^{\ell^r}-1\bigg).
\end{eqnarray*} From the expression \ref{fact2r}, the factorization of $x^{\ell^n}-\mu_k$ becomes: \begin{eqnarray*}
x^{\ell^n}-\mu_k&=&\prod_{i=1}^{\ell^r}(x^{\ell^{n-r}}-\mu_r^i\mu_{k+r})\\&=&\prod_{i=1}^{\ell^r}(x^{\ell^{n-r}}-\mu_{k+r}^{\ell^ki+1}).
\end{eqnarray*} 
If $n\le s-k$, then $r=n$ and hence each factor of $x^{\ell^n}-\mu_k$ is linear over $\mathbb{F}_q$. And  if $n\ge s-k$, then $r=s-k$. It follows that the order of $\mu_{k+r}^{\ell^ki+1}$ is $\ell^{s}$. Therefore, by Lemma \ref{bir},  all binomials $x^{\ell^{n-r}}-\mu_{k+r}^{\ell^ki+1}$ are irreducible over $\mathbb{F}_q$ except the case $s=1$ i.e., $q\equiv3\pmod 4$ when $\ell=2$. Consider the exceptional case, i.e.,   $s=1$ and $\ell=2$, then $x^{2^n}-\mu_k=x^{2^n}+1$. It is well-known that it is irreducible for $n=1$ and reducible for $n\ge 2$ over $\mathbb{F}_q$, where $q\equiv3\pmod 4$.\end{proof} 
\begin{cor} Let $q$ be odd and $r=\min\{n,s-k\}$ and $1\le k\le s$. Then a factorization of $x^{2^n}-\mu_k$ over $\mathbb{F}_q$ is given by:\begin{eqnarray*}x^{2^n}-\mu_k=
\prod_{i=1}^{2^{r}}(x^{2^{n-r}}-\mu_{k+r}^{2^ki+1}).\end{eqnarray*} Further, if $s\neq 1$, then the factorization is  irreducible over $\mathbb{F}_q$. Also if $s=1$, then the irreducible factorization of $x^{2^n}-\mu_k$ is given in Lemma \cite[Lemma 2.6]{singh}.\end{cor}
\begin{proof} The factorization of $x^{2^n}-\mu_k$ follows from Theorem \ref{bifactq} immediately by substituting $\ell=2$.    \end{proof}
 \begin{rem} For each fixed $1\le k\le s$,  the factorization of $x^{\ell^n}-\mu_k$ given in Theorem \ref{bifactq} is also valid every odd prime $\ell$ and $\gcd(\ell,q)=1$.  In particular,  the binomial $x^{\ell^n}-\mu_s$ is irreducible over $\mathbb{F}_q$. \end{rem}  
\begin{Exaexa}\label{exfact163}   Let  $q=163$ and $n=5$. Then $\ell=3$ and  $s=4$. Observe that $\mu_4=-18$, $\mu_3=36$, $\mu_2=38$, $\mu_1=104$.   
The explicit   factorization of $x^{243}-\mu_k$ is given by Lemma  \ref{bifactq} and Table \ref{t243} for each $0\le k\le 3$. 
\begin{table}\label{t243}\begin{center} Table \ref{t243}\\ Parameter of the factorization $x^{243}-\mu_k$  \\
\begin {tabular}{|cccccc|}\hline
$k$&$r=4-k$ &$\mu_k$ & $degree$&coefficients in $\mathbb{F}_{163}^*$&\\&&&$3^{5-r}$&$\{(-18)^{3^{k}i+1}:1\le i\le 3^r\}$& \\\hline
$1$&$3$& $104$ &$9$&$\{(36)^i(-18):1\le i\le 27\}$&\\ 
$2$&$2$& $38$ &$27$&$\{(38)^i(-18):1\le i\le 9\}$&\\
$3$&$1$& $36$ &$81$&$\{(104)^i(-18):1\le i\le 3\}$&\\
$4$&$0$& $-18$ &$243$&$\{(-18)\}$& \\
 \hline
\end{tabular}\end{center}\end{table}
From Table \ref{t243},  the factorization of $x^{243}-38$ is given by: \begin{eqnarray*}x^{243}-38&=&\prod_{i=1}^{3^2}(x^{3^{5-2}}-\mu_{4}^{9i+1})\\&=&\prod_{i=1}^{9}(x^{27}-(-18)^{9i+1})=\prod_{i=1}^{9}(x^{27}-38^i\cdot(-18))\\&=&(x^{27}+32)(x^{27}+75)(x^{27}+79)(x^{27}+68)(x^{27}-24)\\&&(x^{27}+66)(x^{27}+63)(x^{27}-51)(x^{27}+18).\end{eqnarray*}
\end{Exaexa}
\begin{Exaexa} \label{wq97}  Let  $q=97$, $n=8$ and $k=2$. Then $s=5$ and  $r=s-k=3$. Observe that $\mu_5=42$ and $\mu_2=22$.   By Lemma \ref{bifactq},  the factorization of $x^{256}-22$ is given as:  \begin{eqnarray*}x^{256}-22&=&\prod_{i=1}^{8}(x^{32}-\mu_{5}^{4i+1})=\prod_{i=1}^{4}(x^{32}\pm\mu_{5}^{4i+1})\\&=&\prod_{i=1}^{4}(x^{32}\pm42^{4i+1})=\prod_{i=1}^{4}(x^{32}\pm33^i\times42)\\&=&(x^{32}\pm28)(x^{32}\pm46)(x^{32}\pm34)(x^{32}\pm42).\end{eqnarray*}
\end{Exaexa}

\section{Constacyclic codes and their weight distributions}
In this section,  we obtain the explicit expressions of generator polynomials and  the form of codewords of  all  irreducible $\mu$-constacyclic codes of length $\ell^n$ over $\mathbb{F}_q$, where $\ell$ is a prime. Further,  we  also obtain the weight distribution of $\mu$-constacylic codes of length $\ell^n$. 

 Let $\ell$ be a prime (not necessarily odd prime) and $\ell^k|(q-1)$ for $k\ge 1$, and $4|(q-1)$ if $\ell=2$. The explicit factorization of $x^{\ell^n}-\mu_k$ is given in Theorem \ref{bifactq}. Recall  that $r=\min\{n,s-k\}$, where $1\le k\le s$, $s\neq 1$ if $\ell=2$ and $n\ge1$. Then, for each $1\le i\le \ell^r$, in view of Theorem \ref{bifactq},  $x^{\ell^{n-r}}-\mu^{\ell^ki+1}_{{k+r}}$ is a divisor of $x^{\ell^n}-\mu_{k}$. Let $C_{k,i}$ be an $\mu_k$-constacyclic $[\ell^n,\ell^{n-r}]_q$ code over $\mathbb{F}_q$ with the parity check polynomial $x^{\ell^{n-r}}-\mu^{\ell^ki+1}_{{k+r}}$ for $1\le i\le \ell^r$. Then  $$\displaystyle{C_{k,i}=\bigg\langle\frac{x^{\ell^{n}}-\mu_{k}}{x^{\ell^{n-r}}-\mu^{\ell^ki+1}_{k+r}}}\bigg\rangle$$   is a minimal ideal in $\mathbb{F}_q[x]/\langle x^{\ell^n}-\mu_k\rangle$.  \begin{lem}\label{equiv}  For each fixed $1\le k\le s$, there are $\ell^{r}$ equivalent  $\mu_k$-constacyclic $[\ell^n,\ell^{n-r}]_q$ codes, where $\ell$ is a prime and $r=\min\{n,s-k\}$.  \end{lem}
\begin{proof} First we shall show that there are precisely $\ell^r$ constacylic code $C_{k,i}$ for $1\le i\le \ell^r$. Let $C_{k,i}=C_{k,j}$ for any $1\le i,j\le \ell^r$. Then, by definition of $C_{k,i}$, we have $\mu_{k+r}^{\ell^ki+1}=\mu_{k+r}^{\ell^kj+1}$. It follows that $\mu_{k+r}^{\ell^k(i-j)}=1$ for  any $1\le i,j\le \ell^r$. It gives $\ell^r|(i-j)$ for any $1\le i,j\le \ell^r$ and  hence $i=j$.  This proves that $C_{k,i}$ are distinct $\mu_k$-constacyclic codes  for $1\le i\le \ell^r$. 
  
Now, since the order of $\mu_{k+r}^{\ell^ki+1}$ is $\ell^{k+r}$, so the order of  the binomial $x^{\ell^n}-\mu_k$ is $\ell^{n+k}$ for $1\le i\le \ell^r$, $r=\min\{n,s-k\}$ and $1\le k\le s$. It follows that $x^{\ell^{n-r}}-\mu_{k+r}^{\ell^ki+1}$ is the minimal polynomial of $\mu_{n+k}^{\ell^ki+1}$ over $\mathbb{F}_q$. Applying  Proposition \ref{qchain} with $\eta_1=\mu_{n+k}^{\ell^ki+1}$ and $\eta_2=\mu_{n+k}^{\ell^{k}j+1}$ for $1\le i\neq j\le \ell^r$, we find that the polynomials $x^{\ell^n}-\mu_{k+r}^{\ell^ki+1}$ and $x^{\ell^n}-\mu_{k+r}^{\ell^kj+1}$ generate equivalent codes. Therefore $C_{k,i}$ and $C_{k,j}$ are equivalent codes for every $1\le i,j\le \ell^r$. \end{proof}

For each fixed $1\le i\le \ell^r$, from Lemma \ref{equiv}, the constacyclic code $C_{k,i}$ equivalent is equivalent to $C_{k,j}$ for every $1\le j\le \ell^r$. In particular,  the code $C_{k,i}$ is equivalent to $C_{k,\ell^r}$. For convenience point of view, we denote $C_{k,\ell^r}$ by $C_k$.   By Lemma \ref{cc}, since equivalent linear codes share the same weight distributions, so it is sufficient to determine the weight distributions of $C_k$, an $\mu_k$-constacyclic code of length $\ell^n$ with the parity check polynomial $x^{\ell^{n-r}}-\mu_{k+r}$, where $1\le k\le s$ and $k\ge2$ if $\ell=2$.  In the next result,  we give the explicit expressions of the generator polynomials and the form of codewords of all irreducible $\mu_k$-constacyclic code $C_k$ of length $\ell^n$ over $\mathbb{F}_q$.

\begin{thm}\label{concode} Let $\ell$ be a prime and $1\le k\le s$ with $s\neq1 $ if $\ell=2$. Then   the generator polynomial $g_k(x)$ of $C_k$ is given by $$g_k(x)=\displaystyle{\sum_{i=0}^{\ell^{r}-1}
\mu_{k+r}^ix^{\ell^{n-r}(\ell^{r}-i-1)}}.$$
 Further, if ${\bf{a}}\in\mathbb{F}_q^{\ell^{n-r}}$ be any message word,  then  $$C_k=\{({\bf{a}}, {\bf{a}}\mu_{k+r}, \cdots, {\bf{a}}\mu_{k+r}^{{\ell^{r}}-1})
\}.$$ \end{thm}
\begin{proof}   Let $r=\min\{n,s-k\}$ and  $1\le k\le s$.  Since $\mu_{k}=\mu_{k+r}^{\ell^{r}}$, therefore, we have \begin{eqnarray*}x^{\ell^n}-\mu_{k}=(x^{\ell^{n-r}})^{\ell^{r}}-\mu_{k+r}^{\ell^{r}}=(x^{\ell^{n-r}}-\mu_{k+r})\big(\sum_{i=0}^{\ell^{r}-1}
\mu_{k+r}^ix^{\ell^{n-r}(\ell^{r}-i-1)}\big).\end{eqnarray*} 
In this way, the generator polynomial $g_k(x)$ of a $\mu_{k}$-constacyclic $[\ell^n,\ell^{n-r}]_q$ code $C_k$ is 
\begin{eqnarray*}g_k(x)=\dfrac{x^{\ell^n}-\mu_{k}}{x^{\ell^{n-r}}-\mu_{k+r}}=\sum_{i=0}^{\ell^{r}-1}
\mu_{k+r}^ix^{\ell^{n-r}(\ell^{r}-i-1)}.\end{eqnarray*}
 Let ${\bf{a}}=(a_0,a_1,\cdots,a_{\ell^{n-r}-1})\in\mathbb{F}_q^{\ell^{n-r}}$  be any message word. Then the corresponding message polynomials can be expressed as ${\bf{a}}(x)=\displaystyle{\sum_{j=0}^{\ell^{n-r}-1}
a_jx^{\ell^{n-r}-j-1}}$. It follows that  the code polynomial of $C_k$ is  \begin{eqnarray*}{\bf{a}}(x)g_k(x)&=&\left(\sum_{j=0}^{\ell^{n-r}-1}
a_jx^{\ell^{n-r}-j-1}\right)\left(\sum_{i=0}^{\ell^{r}-1}{\mu^i_{k+r}}x^{\ell^{n-r}(\ell^{r}-i-1)}\right)\\&=&\sum_{i=0}^{\ell^{r}-1}\mu_{k+r}^i
\left(\sum_{j=0}^{\ell^{n-r}-1}
a_jx^{\ell^{n-r}-j-1}\right)x^{\ell^{n-r}(\ell^{r}-i-1)}\\&=&\sum_{i=0}^{\ell^{r}-1}\sum_{j=0}^{\ell^{n-r}-1}
a_j\mu_{k+r}^ix^{\ell^{n-r}(\ell^r-i)-j-1}.\end{eqnarray*} Further, the $i$th component of codeword $(c_0^{*},c_1^*,\cdots, c_{\ell^{r}-1}^*)$ is  \begin{center}$c_i^*=(a_0\mu_{k+r}^i, a_1\mu_{k+r}^i, \cdots, a_{\ell^{n-r}-1}\mu_{k+r}^i)$. \end{center} If we denote ${\bf{a}}\theta=(a_0\theta, a_1\theta, \cdots, a_{\ell^{n-r}-1}\theta)$ for some $\theta\in\mathbb{F}_q^*$, then \begin{center}$C_k=\{({\bf{a}}, {\bf{a}}\mu_{k+r}, \cdots, {\bf{a}}\mu_{k+r}^{\ell^r-1}):{\bf{a}}\in\mathbb{F}_q^{\ell^{n-r}}\}$. \end{center}  This completes the proof.\end{proof}
 {\it{\bf{Note:}} In Theorem \ref{concode}, if we consider the case $s=1$, then $C_1=\mathbb{F}_q^{\ell^n}$.}

The following theorem determines the weight distribution of $C_k$.
\begin{thm}\label{wconcode}Let $\ell$ be a prime, $r=\min\{n,s-k\}$ and $1\le k\le s$, and $s\neq1 $ if $\ell=2$. Then the weight distribution of  $\mu_{k}$-constcyclic $[\ell^{n},\ell^{n-r}]_q$ code $C_k$ is  \begin{eqnarray*}A_{\ell^{r}j} = \dbinom{\ell^{n-r}}{j}(q-1)^j& \mbox{for} & 0\le j\le \ell^{n-r}
.\end{eqnarray*}\end{thm}
\begin{proof} By Theorem \ref{concode}, the expression of codewords of $\mu_{k}$-constacyclic $[\ell^n,\ell^{n-r}]_q$ code $C_k$ over $\mathbb{F}_q$ is given by $$C_k=\{({\bf{a}}, {\bf{a}}\mu_{k+r}, \cdots, {\bf{a}}\mu_{k+r}^{{\ell^{r}}-1})
\},$$  where ${\bf{a}}=(a_0,a_1,\cdots,a_{\ell^{n-r}-1})\in\mathbb{F}_q^{\ell^{n-r}}$.  Let $j$  denote the number of nonzero symbols in a message word ${\bf{a}}=(a_0,a_1,\cdots,a_{\ell^{n-r}-1})$.  Then $0\le j\le \ell^{n-r}$ and  the weight of any codeword in $C$ is $\ell^{r}j$.  As the number of non-zero elements in $\mathbb{F}_q$ is $q-1$, so the number of combinations of $j$ nonzero symbols in ${\bf{a}}=(a_0,a_1,\cdots,a_{\ell^{n-r}-1})$  is $\left( \begin{array}{c} \ell^{n-r} \\ j \end{array} \right)(q-1)^j$. This completes the proof.\end{proof}

\begin{Exaexa}   Using notations of Example \ref{exfact163} in Theorem \ref{wconcode}, the weight distributions  of $38$-constcyclic $[243,27]_{163}$ code $C$ is  \begin{eqnarray*}A_{{9}j} = \dbinom{27}{j}(162)^j& \mbox{for} & 0\le j\le 27
.\end{eqnarray*} Table \ref{tw243} provides the weight distribution of all $\mu$-constacyclic codes of length $243$ over $\mathbb{F}_{163}$.
\begin{table}\label{tw243}\begin{center} Table \ref{tw243}\\ Parameter of the factorization $x^{243}-\mu_k$  \\
\begin {tabular}{|cccccc|}\hline
$k$&$r=4-k$ &$\mu_k$ & $Weight $&No. of codwords of weight $i$&\\&&&$i=3^rj$&$A_i=\dbinom{3^{5-r}}{j}(162)^j$&\\&&&$0\le j\le 3^{5-r}$&& \\\hline
$1$&$3$& $104$ &$27j$&$\dbinom{9}{j}(162)^j$&\\ &&&&&\\
$2$&$2$& $38$ &$9j$&$\dbinom{27}{j}(162)^j$&\\&&&&&\\
$3$&$1$& $36$ &$3j$&$\dbinom{81}{j}(162)^j$&\\&&&&&\\
$4$&$0$& $-18$ &$0\le j\le 243$&$\dbinom{243}{j}(162)^j$& \\
 \hline
\end{tabular}\end{center}\end{table}

\end{Exaexa}
\begin{Exaexa} Using notations of Example \ref{wq97}, 
the  generator polynomial of a $22$-constacyclic $[256,32]_{97}$ code $C=\bigg\langle\dfrac{x^{256}-22}{x^{32}-42}\bigg\rangle$ is $$x^{224}+42x^{192}+18x^{160}+77x^{128}+33x^{96}+
28x^{64}+12x^{32}+19.$$
  Then  the code polynomial of $C$ is  \begin{eqnarray*}\sum_{i=0}^{7}\sum_{j=0}^{31}
a_j(42)^ix^{256-32i-j-1}.\end{eqnarray*}
 In view of Theorem \ref{concode}, $22$-constacyclic code  of length $256$ over $\mathbb{F}_{97}$ is given by   $$C=\{({\bf{a}}, 42{\bf{a}}, \cdots,(42)^{7}{\bf{a}})
\}$$  with ${\bf{a}}=(a_0,a_1,\cdots, a_{31})\in\mathbb{F}_{97}^{32}$.
 By Theorem \ref{wconcode}, the weight distribution  of $22$-constcyclic $[256,32]_{97}$ code $C$ is  \begin{eqnarray*}A_{{8}j} = \dbinom{32}{j}(96)^j& \mbox{for} & 0\le j\le 32
.\end{eqnarray*}
\end{Exaexa}

\end {document}